\batchmode
\makeatletter
\def\input@path{{"C:/Users/magno/Dropbox/Universita/Ricerche/Contact Curves/Latex/"}}
\makeatother
\documentclass[oneside,english]{amsart}
\usepackage[T1]{fontenc}
\usepackage[latin9]{inputenc}
\usepackage{babel}
\usepackage{array}
\usepackage{verbatim}
\usepackage{float}
\usepackage{textcomp}
\usepackage{mathrsfs}
\usepackage{url}
\usepackage{multirow}
\usepackage{amsthm}
\usepackage{amssymb}
\usepackage[all]{xy}
\usepackage[pdftex,unicode=true,
 bookmarks=true,bookmarksnumbered=false,bookmarksopen=false,
 breaklinks=false,pdfborder={0 0 1},backref=page,colorlinks=false]
 {hyperref}
\hypersetup{
 pdfauthor={Giosuè Muratore}}

\makeatletter

\providecommand{\tabularnewline}{\\}

\numberwithin{equation}{section}
\numberwithin{figure}{section}
\numberwithin{table}{section}
\theoremstyle{plain}
\newtheorem{thm}{\protect\theoremname}[section]
\theoremstyle{definition}
\newtheorem{defn}[thm]{\protect\definitionname}
\theoremstyle{definition}
\newtheorem{example}[thm]{\protect\examplename}
\theoremstyle{remark}
\newtheorem{rem}[thm]{\protect\remarkname}
\theoremstyle{plain}
\newtheorem{prop}[thm]{\protect\propositionname}
\theoremstyle{plain}
\newtheorem{lem}[thm]{\protect\lemmaname}
\theoremstyle{plain}
\newtheorem{cor}[thm]{\protect\corollaryname}
\theoremstyle{remark}
\newtheorem{claim}[thm]{\protect\claimname}

\usepackage{tikz}
\usetikzlibrary{positioning}

\makeatother

\providecommand{\claimname}{Claim}
\providecommand{\corollaryname}{Corollary}
\providecommand{\definitionname}{Definition}
\providecommand{\examplename}{Example}
\providecommand{\lemmaname}{Lemma}
\providecommand{\propositionname}{Proposition}
\providecommand{\remarkname}{Remark}
\providecommand{\theoremname}{Theorem}

\begin{document}
\global\long\def\d{\mathrm{d}}%
\global\long\def\ev{\mathrm{ev}}%
\global\long\def\ctopT{c_{\mathrm{top}}^{T}}%
\global\long\def\ctop{c_{\mathrm{top}}}%
\global\long\def\ch{\mathrm{ch}}%
\global\long\def\td{\mathrm{td}}%
\global\long\def\CSM{\mathbb{S}}%
\global\long\def\f{\delta}%

\title{Enumeration of rational contact curves via torus actions}
\author{Giosuè Muratore}
\date{\today}
\address{Department of Mathematics, Universidade Federal de Minas Gerais, Belo
Horizonte, MG, Brazil.}
\email{\href{mailto:muratore.g.e@gmail.com}{muratore.g.e@gmail.com}}
\urladdr{\url{https://sites.google.com/view/giosue-muratore}}
\begin{abstract}
We prove that some Gromov-Witten numbers associated to rational contact
(Legendrian) curves in any contact complex projective space with arbitrary
incidence conditions are enumerative. Also, we use Bott formula on
the Kontsevich space to find the exact value of those numbers. As
an example, the numbers of rational contact curves of low degree in
$\mathbb{P}^{3}$ and $\mathbb{P}^{5}$ are computed. The results
are consistent with existing results.
\end{abstract}

\maketitle

\section{Introduction}

A contact structure on a smooth complex manifold $X$ of odd dimension
is a corank $1$ non-integrable holomorphic distribution. For example,
projective spaces $\mathbb{P}^{2n+1}$ carry a canonical contact structure.
Roughly speaking, this structure can be defined by the association,
for each point $p\in\mathbb{P}^{2n+1}$, of a hyperplane $H_{p}\subset\mathbb{P}^{2n+1}$
containing $p$. 

Contact structures attracted a lot of interest in recent years, both
from mathematicians and physicist, and their classification is still
ongoing. For example, Ye proved in \cite{ye1994note} that a contact
threefold is either $\mathbb{P}^{3}$ or $\mathbb{P}(T_{M})$ for
some smooth surface $M$. Druel \cite{druel1999structures} proved
that toric contact manifolds are either $\mathbb{P}^{2n+1}$ or $\mathbb{P}(T_{\mathbb{P}^{1}\times\cdots\times\mathbb{P}^{1}})$.
Finally, \cite{druel1998structures} and \cite{kebekus2000projective}
partially classify contact manifolds by the positivity of the canonical
divisor using Mori theory.

A submanifold of a contact manifold $X$ is called contact (or Legendrian)
if it is point-wise tangent to the contact distribution. Contact curves
play a central role in the classification of contact manifolds, as
well as in the theory of superminimal Riemann surfaces (see \cite{bryant1982conformal}).
In this article we are interested in \textbf{rational} contact curves
$C\subset\mathbb{P}^{2n+1}$. Despite their importance, enumerative
invariants of contact curves are largely unknown, even for projective
spaces. The first paper to address this problem is \cite{levcovitz2011symplectic},
where mainly planar curves in $\mathbb{P}^{3}$ are analyzed. But
the first work studying this problem using torus actions is Amorim's
Ph.D. thesis \cite{Eden}. The main result is the computation of the
equivariant class of the class of contact curves in the Kontsevich
moduli space $\overline{M}_{0,0}(\mathbb{P}^{3},d)$. Using Atiyah-Bott
formula, he computes an invariant related to the number of degree
$3$ (resp. $4$) rational contact curves meeting $7$ (resp. $9$)
general lines in $\mathbb{P}^{3}$.

In this article, we will prove that the moduli space of rational contact
stable maps is of the expected dimension (Lemma \ref{lem:ExpDim}).
In particular, Gromov-Witten invariants of rational contact curves
are enumerative (Corollary \ref{cor:Enum}). Further, we will use
torus actions and equivariant cohomology to provide virtually all
enumerative invariants of rational contact curves of degree $d$ in
any projective space meeting an arbitrary number of linear subspaces.
In the last section, we list such invariants for $\mathbb{P}^{3}$
and $\mathbb{P}^{5}$, and low values of $d$ (Table \ref{tab:EnuNum},
\ref{tab:EnuNuminP5} and \ref{tab:EnuNuminP5-con}). More invariants
will be in \cite{MurSch}.

The results of this article can be used to find any enumerative invariants
of rational contact curves. These results are based on Atiyah-Bott
formula, which usually requires to perform tedious computations. I
think that quantum cohomology can be used to find a generating function
giving recursively all enumerative invariants of rational contact
curves. Such generating function would look like the generating function
for usual rational curves in $\mathbb{P}^{3}$, see \cite[Section 9]{FP}.
This may be the topic of a future work.

I would like to thank Israel Vainsencher for introducing to me contact
geometry. I also thank Alex Abreu for a fruitful discussion and Éden
Amorim for explaining me some part of his thesis. The author is supported
by postdoctoral fellowship PNPD-CAPES.

\section{Contact structures}

In this section we recall basic facts about complex contact structures.
All varieties are defined over $\mathbb{C}$.
\begin{defn}[\cite{kobayashi1959remarks}]
\label{def:ContactStructure}Let $X$ be a $(2n+1)$-dimensional
complex manifold. A contact structure on $X$ is a open cover $\{U_{i}\}_{i}$
of $X$ together with $1$-forms $\alpha_{i}$ on each $U_{i}$ such
that:
\begin{enumerate}
\item at every point of $U_{i}$, the form $\alpha_{i}\wedge(\d\alpha_{i})^{\wedge n}$
is non-zero, and
\item if the intersection $U_{i}\cap U_{j}$ is non empty, then there exists
a non-vanishing holomorphic function $f_{ij}$ on $U_{i}\cap U_{j}$
such that $\alpha_{i}=f_{ij}\alpha_{j}$ on $U_{i}\cap U_{j}$.
\end{enumerate}
\end{defn}

\begin{defn}
\label{def:ContactCurve}Let $X$ be a manifold with a contact structure
$\{U_{i},\alpha_{i}\}$. A curve $C\subset X$ is a contact curve
if any local $1$-form $\alpha_{i}$ vanishes at every smooth point
of $C\cap U_{i}$.
\end{defn}

All odd-dimensional projective spaces $\mathbb{P}^{2n+1}=\mathbb{P}(\mathbb{C}^{2n+2})$
are contact. Indeed, let $\{x_{i}\}_{i=0}^{2n+1}$ be local coordinates
in $\mathbb{C}^{2n+2}$, and let
\begin{equation}
\alpha:=\sum_{i=0}^{n}x_{2i}\d x_{2i+1}-x_{2i+1}\d x_{2i}\label{eq:1Form}
\end{equation}
be a $1$-form on $\mathbb{C}^{2n+2}\backslash\{0\}$. Let $\{U_{i}\}_{i}$
be an open cover of $\mathbb{P}^{2n+1}$, and for every index $i$
let $r_{i}$ be a holomorphic section over $U_{i}$ of the standard
projection $\mathbb{C}^{2n+2}\backslash\{0\}\rightarrow\mathbb{P}^{2n+1}$.
Then $\{U_{i},r_{i}^{*}\alpha\}$ is a contact structure. The line
given by the equations
\[
x_{i}=0,\,i=1,\ldots,2n
\]
is an example of a contact curve. The moduli space of contact lines
in $\mathbb{P}^{2n+1}$ has a very explicit description.
\begin{example}
\label{exa:SL}Let us consider the symplectic form $\d\alpha\in\wedge^{2}(\mathbb{C}^{2n+2})^{\vee}$.
Let $\mathcal{S}$ be the tautological rank $2$ vector bundle on
the Grassmannian $G:=G(2,2n+2)$. The canonical inclusion $\mathcal{S}\subset\mathbb{C}^{2n+2}\times\mathcal{O}_{G}$
induces a surjective map:
\[
\wedge^{2}(\mathbb{C}^{2n+2})^{\vee}\times\mathcal{O}_{G}\twoheadrightarrow\wedge^{2}\mathcal{S}^{\vee},
\]
where $\wedge^{2}\mathcal{S}^{\vee}=\mathcal{O}_{G}(1)$ is the Plücker
line bundle. Hence $\d\alpha$ induces a global section of $\mathcal{O}_{G}(1)$,
whose zero locus is exactly the locus of $2$-dimensional subspaces
of $\mathbb{C}^{2n+2}$ isotropic with respect to $\d\alpha$. Such
locus is usually denoted by $SG(2,2n+2)$ and parameterizes contact
projective lines in $\mathbb{P}^{2n+1}$.
\end{example}

\begin{rem}
It is not hard to see that the definition of contact manifold is equivalent
to the following: A contact structure on $X$ is a pair $(X,L)$ where
$L$ is a line subbundle of $\Omega_{X}^{1}$ such that if $s$ is
a non trivial local section of $L$, then $s\wedge(\d s)^{\wedge n}$
is everywhere non-zero. So, contact curves are exactly those curves
$C$ such that $s_{|T_{C}}=0$.
\end{rem}

\begin{example}
\label{exa:ContactOddDim}Any symplectic form on $\mathbb{C}^{2n+2}$
induces in a natural way a subbundle $\mathcal{O}_{\mathbb{P}^{2n+1}}(-2)\hookrightarrow\Omega_{\mathbb{P}^{2n+1}}^{1}$.
All contact structures of $\mathbb{P}^{2n+1}$ can be constructed
this way, and they are equivalent by a change of coordinates, see
\cite[1.4.2]{okonek1980vector}.
\end{example}

\begin{prop}
\label{prop:EqDefs}Let $X$ be a manifold with a contact structure
$\{U_{i},\alpha_{i}\}_{i}$, and let $C\subset X$ be a curve. Let
$L$ be the contact line bundle and let $\mathcal{D}:=\ker(T_{X}\twoheadrightarrow L^{\vee})$.
The following are equivalent.
\begin{itemize}
\item $C$ is a contact curve.
\item $T_{C,p}\subset\mathcal{D}_{p}$ for every smooth point $p\in C$.
\end{itemize}
Moreover if $p\in C\cap U_{i}$, then $T_{C,p}$ is an isotropic vector
subspace of the symplectic space $(\mathcal{D}_{p},\d\alpha_{i})$.
\end{prop}

\begin{proof}
By Definition \ref{def:ContactStructure}, it follows that for every
$1$-form $\alpha$ the $2$-form $\d\alpha$ restricted to $\mathcal{D}_{p}$
is a symplectic form.

Let $i:C\rightarrow X$ be the inclusion. The condition for $C$ to
be contact becomes $i^{*}\alpha_{i}\equiv0$, that is $T_{C,p}\subset\mathcal{D}_{p}$.
On the other hand, if $T_{C,p}\subset\mathcal{D}_{p}$ then $C$ is
contact. Finally, $i^{*}\alpha\equiv0$ implies $i^{*}\d\alpha\equiv0$
so that $T_{C,p}$ is an isotropic subspace of $(\mathcal{D}_{p},\d\alpha_{i})$.
\end{proof}

\section{Contact stable maps}
\begin{defn}
A stable map of $\mathbb{P}^{n}$ is the datum $(C,f,p_{1},\ldots,p_{m})$
where $C$ is a projective, connected, nodal curve of arithmetic genus
$0$, the markings $p_{1},\ldots,p_{m}$ are distinct nonsingular
points of $C$, and $f:C\rightarrow\mathbb{P}^{n}$ is a morphism
such that $f_{*}([C])$ is a $1$-cycle of degree $d$. Moreover,
for every rational component $E\subset C$ mapped to a point, $E$
contains at least three points among marked points and nodal points.
Two stable maps, $(C,f,p_{1},\ldots,p_{m})$ and $(C',f',p'_{1},\ldots,p'_{m})$,
are equivalent if there exists an isomorphism $\varphi:C\rightarrow C'$
such that $f=f'\circ\varphi$ and $\varphi(p_{i})=p'_{i}$ for all
$i=1,\ldots,m$.
\end{defn}

For any non-negative integer $m$, we denote by $\overline{\mathcal{M}}_{0,m}(\mathbb{P}^{n},d)$
the moduli stack of stable maps of degree $d$, with coarse moduli
space $\overline{M}_{0,m}(\mathbb{P}^{n},d)$. Moreover, we denote
by
\[
\xymatrix{\mathcal{U}_{m}\ar[d]^{\pi}\ar[r]^{\ev_{m,0}} & \mathbb{P}^{n}\\
\overline{\mathcal{M}}_{0,m}(\mathbb{P}^{n},d)
}
\]
the universal family. For $i=1,\ldots,m$ we denote by $\ev_{i}:\overline{M}_{0,m}(\mathbb{P}^{n},d)\rightarrow\mathbb{P}^{n}$
the evaluations maps. Further, there are two interesting subschemes
of $\overline{M}_{0,m}(\mathbb{P}^{n},d)$. They are $M_{0,m}(\mathbb{P}^{n},d)$
and $\overline{M}_{0,m}^{*}(\mathbb{P}^{n},d)$, with the following
properties:
\begin{itemize}
\item $M_{0,m}(\mathbb{P}^{n},d)$ parameterizes stable maps with irreducible
domain,
\item $\overline{M}_{0,m}^{*}(\mathbb{P}^{n},d)$ is smooth and it is a
fine moduli space of automorphisms-free stable maps.
\end{itemize}
Both of them are open and dense in $\overline{M}_{0,m}(\mathbb{P}^{n},d)$.
See \cite{FP} for more details.

As mentioned in Example \ref{exa:ContactOddDim}, the space $\mathbb{P}^{2n+1}$
has an essentially unique contact structure with contact line bundle
$L=\mathcal{O}_{\mathbb{P}^{2n+1}}(-2)$.
\begin{defn}
Let $C$ be a projective, connected, nodal curve of arithmetic genus
$0$. A contact morphism is a morphism $f:C\rightarrow\mathbb{P}^{2n+1}$
such that for any local section $s$ of $L$, the map $f^{*}s$ vanishes
at every point of $C$. 
\end{defn}

A morphism $\mathbb{P}^{1}\rightarrow\mathbb{P}^{2n+1}$ of degree
$d$ is given by $2n+2$ sections of $H^{0}(\mathbb{P}^{1},\mathcal{O}_{\mathbb{P}^{1}}(d))$,
modulo a multiplicative constant. That is, a point
\[
[s_{0}:\ldots:s_{2n+1}]\in\mathbb{P}(V),
\]
where $V$ is the $(2n+2)$-fold product of $H^{0}(\mathbb{P}^{1},\mathcal{O}_{\mathbb{P}^{1}}(d))$.
Such morphism is contact exactly when the pull back of the $1$-form
(\ref{eq:1Form}) vanishes. That is, when $\sum_{i=0}^{n}s_{2i}\d s_{2i+1}-s_{2i+1}\d s_{2i}=0$.
We will use the following result of Kobak and Loo.
\begin{prop}
\label{prop:KL}Let $\mathring{\mathscr{M}}_{n,d}\subset\mathbb{P}(V)$
be the moduli space of contact morphisms $\mathbb{P}^{1}\rightarrow\mathbb{P}^{2n+1}$
of degree $d$. Then $\mathring{\mathscr{M}}_{n,d}$ is a connected
quasi-projective reduced variety of dimension $2n(d+1)+2$.
\end{prop}

\begin{proof}
See \cite[PROPOSITION 4.3]{KL}.
\end{proof}
Note that the property of being contact is clearly invariant by automorphisms
of $C$.
\begin{defn}
A contact stable map is a stable map $(C,f,p_{1},\ldots,p_{m})$ of
$\mathbb{P}^{2n+1}$ such that $f:C\rightarrow\mathbb{P}^{2n+1}$
is a contact morphism for some (hence, any) representative of $(C,f,p_{1},\ldots,p_{m})$.
We denote by $\CSM_{m}(d)$ the space of all contact stable maps of
degree $d$ and $m$ marks, with canonical inclusion $j:\CSM_{m}(d)\rightarrow\overline{M}_{0,m}(\mathbb{P}^{2n+1},d)$.
For every $i=1,\ldots,m$, we have maps $\tilde{\ev}_{i}:\CSM_{m}(d)\rightarrow\mathbb{P}^{2n+1}$
given by composition $\tilde{\ev}_{i}:=\ev_{i}\circ j$.
\end{defn}

\begin{rem}
If $C$ is a contact curve contained in $\mathbb{P}^{2n+1}$, and
$f:C\rightarrow\mathbb{P}^{2n+1}$ is the embedding, then $(C,f,p_{1},\ldots,p_{m})$
is a contact stable map. Anyway, not all contact stable maps are of
this form. For example, if $f$ is a finite cover of a contact curve,
then $(C,f,p_{1},\ldots,p_{m})$ is a contact stable map.
\end{rem}

The moduli space of contact stable maps in $\mathbb{P}^{2n+1}$ is
the zero locus of a section of a vector bundle on $\overline{M}_{0,m}(\mathbb{P}^{2n+1},d)$,
that is:
\begin{thm}
\label{thm:Moduli}$\CSM_{m}(d)$ is the zero locus of a section of
\[
\mathcal{E}_{d,m}:=\pi_{*}(\omega_{\pi}\otimes\ev_{m,0}^{*}L^{\vee}).
\]
In particular, the expected dimension of $\CSM_{m}(d)$ is $2n(d+1)-1+m$.
\end{thm}

\begin{proof}
We follow the approach of \cite[3.2]{levcovitz2011symplectic}. Let
$(C,f,p_{1},\ldots,p_{m})$ be a stable map. By Proposition \ref{prop:EqDefs}
$(C,f,p_{1},\ldots,p_{m})$ is contact if and only if the map $T_{C}\rightarrow f^{*}T_{\mathbb{P}^{2n+1}}$
factors through $f^{*}\mathcal{D}$, i.e., the composition $T_{C}\rightarrow f^{*}L^{\vee}$
is the zero map. The locus of such maps is the locus where the map
$T_{\pi}\rightarrow\ev_{n,0}^{*}L^{\vee}$ vanishes along the whole
fiber of $\pi$ over $(C,f,p_{1},\ldots,p_{m})$. The sheaf $\mathcal{E}_{d,m}$
is locally free by \cite[III.12.9]{MR0463157} as $h^{1}(C,\omega_{C}\otimes f^{*}L^{\vee})=h^{0}(C,f^{*}L)=0$.
It follows that the locus of contact stable maps is given by a section
of $\mathcal{E}_{d,m}$ by \cite[(2.3)]{altman1977foundations}. Note
that the argument in Levcovitz-Vainsencher is made for $m=0$. In
the case $m>0$ we may have some component $C'\subset C$ contracted
by $f$. But in that case the map $T_{C'}\rightarrow f^{*}T_{\mathbb{P}^{2n+1}}$
is zero, so the result is still valid.

Since the rank of $\mathcal{E}_{d,m}$ is $h^{0}(\mathbb{P}^{1},\omega_{\mathbb{P}^{1}}\otimes f^{*}L^{\vee})=h^{0}(\mathbb{P}^{1},\mathcal{O}_{\mathbb{P}^{1}}(2d-2))=2d-1$,
the expected dimension of $\CSM_{m}(d)$ is $\dim\overline{M}_{0,m}(\mathbb{P}^{2n+1},d)-(2d-1)=2n(d+1)-1+m$.
\end{proof}
The following results are useful for enumerative applications.
\begin{lem}
\label{lem:ExpDim}The moduli space $\CSM_{m}(d)$ is purely of the
expected dimension.
\end{lem}

\begin{proof}
Let $X=\mathbb{P}^{2n+1}$. If $\f:\overline{M}_{0,m+1}(X,d)\rightarrow\overline{M}_{0,m}(X,d)$
is the forgetful morphism of the last mark, we see that for every
$m\ge0$ we have $\CSM_{m+1}(d)=\f^{-1}(\CSM_{m}(d))$. Indeed, let
$(D,g,p_{1},\ldots,p_{m},p_{m+1})\in\overline{M}_{0,m+1}(X,d)$ and
\[
\f((D,g,p_{1},\ldots,p_{m},p_{m+1}))=(C,f,p_{1},\ldots,p_{m})\in\overline{M}_{0,m}(X,d).
\]
By construction, there exists a map $\hat{\f}:D\rightarrow C$ such
that $g=f\circ\hat{\f}$ and $\hat{\f}$ is, in every component of
$D$, either a contraction or an isomorphism. If $s$ is any local
section of $L$, $g^{*}s$ vanishes at some point if and only if $\hat{\f}^{*}(f^{*}s)$
vanishes. It follows that $(D,g,p_{1},\ldots,p_{m},p_{m+1})$ is contact
if and only if $(C,f,p_{1},\ldots,p_{m})$ is contact. Since all fibers
of the forgetful map are connected and $1$-dimensional, it follows
that if $\CSM_{0}(d)$ is purely of the expected dimension for some
$d$, then the same holds for $\CSM_{m}(d)$ for all $m\ge0$.

We will prove the theorem by induction on $d$. We know that $\CSM_{0}(1)$
is purely of the expected dimension by Example \ref{exa:SL}.

Since $\CSM_{m}(d)$ is the zero section of a vector bundle, each
component is at least of the expected dimension. Consider the moduli
space $\mathring{\mathscr{M}}_{n,d}$ of Proposition \ref{prop:KL}.
By the universal property of $\overline{M}_{0,0}(X,d)$, there exists
a map $\mathring{\mathscr{M}}_{n,d}\rightarrow\overline{M}_{0,0}(X,d)$
which sends a contact morphism $f:\mathbb{P}^{1}\rightarrow X$ to
the contact stable map $(\mathbb{P}^{1},f)$. We denote the image
by $\CSM_{0}(d)^{\circ}$. It is clear that $\CSM_{0}(d)^{\circ}$
is the scheme of all contact stable maps with irreducible domain,
i.e., $\CSM_{0}(d)^{\circ}=\CSM_{0}(d)\cap M_{0,0}(X,d)$. Moreover,
the fiber of any point $(\mathbb{P}^{1},f)\in\CSM_{0}(d)^{\circ}$
is the orbit 
\begin{equation}
\{\varphi\circ f:\mathbb{P}^{1}\rightarrow X\,/\,\forall\varphi\in\mathrm{Aut}(\mathbb{P}^{1})\}.\label{eq:Orbit}
\end{equation}
Note that (\ref{eq:Orbit}) is isomorphic to $\mathrm{Aut}(\mathbb{P}^{1})$
when $(\mathbb{P}^{1},f)$ is automorphisms-free. Otherwise, it is
a quotient of $\mathrm{Aut}(\mathbb{P}^{1})$ by a finite group. It
follows that the dimension of each component of $\CSM_{0}(d)^{\circ}$
is at most $2n(d+1)+2-3=2n(d+1)-1$, as expected. So, the closure
of $\CSM_{0}(d)^{\circ}$ is purely of the expected dimension.

The space $\CSM_{0}(d)$ may have components supported in the complement
of $M_{0,0}(X,d)$. That complement is the union of boundary divisors
of the form
\begin{equation}
\overline{M}_{0,1}(X,d_{1})\times_{X}\overline{M}_{0,1}(X,d_{2}),\label{eq:strata}
\end{equation}
for $d_{1}$ and $d_{2}$ positive integers such that $d_{1}+d_{2}=d$.
Since the property of being contact is a local condition, a stable
map contained in the variety of display (\ref{eq:strata}) is contact
if and only if each component is contact. It follows that the moduli
space of contact stable maps contained in (\ref{eq:strata}) is the
fiber product:
\[
\CSM_{1}(d_{1})\times_{X}\CSM_{1}(d_{2}).
\]
By induction, it has dimension:
\begin{eqnarray*}
2n(d_{1}+1)+2n(d_{2}+1)-\dim X & = & 2n(d_{1}+d_{2}+1)+2n-2n-1\\
 & = & 2n(d+1)-1.
\end{eqnarray*}
That is, the same as $\CSM_{0}(d)^{\circ}$. So $\CSM_{0}(d)$ is
purely of the expected dimension.
\end{proof}
There exist results in other contexts similar to Theorem \ref{thm:Moduli}
and Lemma \ref{lem:ExpDim}. See \cite[Proposition 5.1]{bryant1991two}
and \cite[2.3\&2.4]{YeArxiv}.
\begin{thm}
Let $\Gamma_{1},\ldots,\Gamma_{m}$ be subvarieties of $\mathbb{P}^{2n+1}$,
such that $\sum_{i=1}^{m}\mathrm{codim}(\Gamma_{i})=2n(d+1)-1+m$.
Let $G$ be the group of invertible complex matrices of order $2n+2$.
Then for a general $\sigma=(g_{1},\ldots,g_{m})\in G^{\times m}$,
the scheme-theoretic intersection
\begin{equation}
\CSM_{m}(d)\cap\ev_{1}^{-1}(g_{1}\Gamma_{1})\cap\cdots\cap\ev_{m}^{-1}(g_{m}\Gamma_{m})\label{eq:Intersection}
\end{equation}
is a finite number of reduced points supported in $\overline{M}_{0,m}^{*}(\mathbb{P}^{2n+1},d)$.
Further, the number of points in this intersection is
\begin{equation}
\int_{\overline{M}_{0,m}(\mathbb{P}^{2n+1},d)}\ev_{1}^{*}(\Gamma_{1})\cdots\ev_{m}^{*}(\Gamma_{m})\cdot\ctop(\mathcal{E}_{d,m}).\label{eq:The1}
\end{equation}
\end{thm}

\begin{proof}
We may apply the celebrated Kleiman-Bertini Theorem \cite{Kle}: Let
$G$ be a connected algebraic group, and $X$ a homogeneous $G$-variety,
$Y$ and $Z$ varieties mapping to $X$:
\[
\xymatrix{ & Z\ar[d]^{\beta}\\
Y\ar[r]^{\alpha} & X
}
\]
For $\sigma\in G$, denote by $Y^{\sigma}$ the image of the composition
$Y\stackrel{\alpha}{\rightarrow}X\stackrel{\sigma}{\rightarrow}X$.
Then
\begin{enumerate}
\item There is a dense open subscheme $G^{\circ}$ of $G$ such that, for
$\sigma\in G^{\circ}$, $Y^{\sigma}\times_{X}Z$ is either empty,
or of pure dimension
\[
\dim Y+\dim Z-\dim X.
\]
\item Further, if $Y$ and $Z$ are nonsingular, then $G^{\circ}$ can be
found so that $Y^{\sigma}\times_{X}Z$ is nonsingular.
\end{enumerate}
Let us denote $X=\mathbb{P}^{2n+1}$, which is a homogeneous $G$-variety
with respect to the group of invertible matrices of order $2n+2$.
If we apply Kleiman-Bertini to the diagram
\[
\xymatrix{ & \CSM_{m}(d)\ar[d]^{(\tilde{\ev}_{1}\times\cdots\times\tilde{\ev}_{m})}\\
\Gamma_{1}\times\cdots\times\Gamma_{m}=:\Gamma\ar[r] & X^{\times m}
}
\]
we deduce that, for a general $\sigma\in G^{\times m}$, $\Gamma^{\sigma}\times_{X^{\times m}}\CSM_{m}(d)$
is either empty, or of dimension $0$. We used that $\CSM_{m}(d)$
has the expected dimension. We assume, without loss of generality,
that $\Gamma^{\sigma}$ is smooth by applying again Kleiman-Bertini.
In the same way we may prove\footnote{Details can be found in \cite[Lemma p. 5]{Tho98} or \cite[Lemma 14]{FP}.}
that the intersection in (\ref{eq:Intersection}) is supported in
the intersection between $\overline{M}_{0,m}^{*}(X,d)$ and the smooth
part of $\CSM_{m}(d)$.

So, when $\Gamma^{\sigma}\times_{X^{\times m}}\CSM_{m}(d)$ is not
empty, it is reduced of dimension $0$. Standard arguments of intersection
theory tell us that (\ref{eq:Intersection}) is the degree of the
zero cycle
\begin{equation}
\int_{\CSM_{m}(d)}\tilde{\ev}_{1}^{*}(\Gamma_{1})\cdots\tilde{\ev}_{m}^{*}(\Gamma_{m}),\label{eq:The2}
\end{equation}
in the Chow ring of $\CSM_{m}(d)$. The equivalence between (\ref{eq:The1})
and (\ref{eq:The2}) follows from the fact that $j_{*}[\CSM_{m}(d)]=\ctop(\mathcal{E}_{d,m})$
by Theorem \ref{thm:Moduli}.
\end{proof}
\begin{cor}
\label{cor:Enum}Let $\Gamma_{1},\ldots,\Gamma_{m}$ be subvarieties
of $\mathbb{P}^{2n+1}$ in general position, such that $\sum_{i=1}^{m}\mathrm{codim}(\Gamma_{i})=2n(d+1)-1+m$.
Then the number of contact curves in $\mathbb{P}^{2n+1}$ of degree
$d$ meeting all $\Gamma_{1},\ldots,\Gamma_{m}$ is
\begin{equation}
\int_{\overline{M}_{0,m}(\mathbb{P}^{2n+1},d)}\ev_{1}^{*}(\Gamma_{1})\cdots\ev_{m}^{*}(\Gamma_{m})\cdot\ctop(\mathcal{E}_{d,m}).\label{eq:Enum}
\end{equation}
\end{cor}

\begin{proof}
The zero cycle (\ref{eq:Enum}) equals the number of contact stable
maps $(C,f,p_{1},\ldots,p_{m})$ satisfying $f(p_{i})\in g_{i}\Gamma_{i}$
for some general $g_{i}$. The property of being in general position
is preserved by the action of a general $\sigma\in G^{\times m}$,
so we can substitute all $\Gamma_{i}$ with $g_{i}\Gamma_{i}$. Since
these stable maps are automorphisms-free, the set of all curves $f(C)\subset\mathbb{P}^{2n+1}$
is the set of all contact curves meeting all $\Gamma_{1},\ldots,\Gamma_{m}$.
\end{proof}
\begin{rem}
We have seen that $\CSM_{m}(d)$ is purely of the expected dimension,
but it is far from being irreducible. Indeed, the components of $\CSM_{m}(d)$
are nasty fiber products
\[
\CSM_{m_{1}}(d_{1})\times_{\mathbb{P}^{2n+1}}\cdots\times_{\mathbb{P}^{2n+1}}\CSM_{m_{k}}(d_{k}),
\]
where $(d_{1},\ldots,d_{k})$ is a partition of $d$. Since all of
them have the same dimension, all of them have non trivial contribution
to (\ref{eq:The1}). So Equation (\ref{eq:The1}) considers all irreducible
contact curves as well as all reducible ones.
\end{rem}

We can reduce to the case $m=0$ thanks to the following
\begin{lem}
\label{lem:WoutMarks}Let $\Gamma_{1},\ldots,\Gamma_{m}$ be subvarieties
of $X=\mathbb{P}^{2n+1}$, such that $\sum_{i=1}^{m}\mathrm{codim}(\Gamma_{i})=2n(d+1)-1+m$.
Let $\f:\overline{M}_{0,1}(X,d)\rightarrow\overline{M}_{0,0}(X,d)$
be the forgetful map. Then
\begin{equation}
\begin{array}{c}
\int_{\overline{M}_{0,m}(X,d)}\ev_{1}^{*}(\Gamma_{1})\cdots\ev_{m}^{*}(\Gamma_{m})\cdot\ctop(\mathcal{E}_{d,m})\\
\,\,\,\,\,\,\,\,\,\,\,\,\,\,\,\,\,\,\,\,\,\,\,\,\,\,\,\,\,\,\,\,\,\,\,\,\,\,\,\,\,\,\,\,\,\,\,\,\,\,\,\,\,\,=\int_{\overline{M}_{0,0}(X,d)}\f_{*}(\ev_{1}^{*}(\Gamma_{1}))\cdots\f_{*}(\ev_{1}^{*}(\Gamma_{m}))\cdot\ctop(\mathcal{E}_{d,0})
\end{array}\label{eq:WoutMarks}
\end{equation}
\end{lem}

\begin{proof}
Suppose $m=1$ and consider the natural diagram of stacks
\[
\xymatrix{\mathcal{U}_{1}\ar[d]_{\pi'}\ar[r]^{\tilde{\f}} & \mathcal{U}_{0}\ar[d]^{\pi}\ar[r]^{\ev_{0,0}} & X\\
\overline{\mathcal{M}}_{0,1}(X,d)\ar[r]^{\f} & \overline{\mathcal{M}}_{0,0}(X,\beta)
}
\]
where $\tilde{\f}$ is the natural map such that $\ev_{1,0}=\ev_{0,0}\circ\tilde{\f}$.
The square is clearly Cartesian, so that $\tilde{\f}^{*}\omega_{\pi}=\omega_{\pi'}$
and $\f^{*}\mathcal{E}_{d,0}=\mathcal{E}_{d,1}$. By \cite[(3.9)Lemma]{vistoli1989intersection}
and projection formula we have: 
\begin{eqnarray*}
\int_{\overline{\mathcal{M}}_{0,1}(X,d)}\ev_{1}^{*}(\Gamma_{1})\cdot\ctop(\mathcal{E}_{d,1}) & = & \f_{*}\int_{\overline{\mathcal{M}}_{0,1}(X,d)}\ev_{1}^{*}(\Gamma_{1})\cdot\ctop(\f^{*}\mathcal{E}_{d,0})\\
 & = & \f_{*}\int_{\f^{*}[\overline{\mathcal{M}}_{0,0}(X,d)]}\ev_{1}^{*}(\Gamma_{1})\cdot\f^{*}\ctop(\mathcal{E}_{d,0})\\
 & = & \int_{\overline{\mathcal{M}}_{0,0}(X,d)}\f_{*}(\ev_{1}^{*}(\Gamma_{1}))\cdot\ctop(\mathcal{E}_{d,0})
\end{eqnarray*}
Since the intersection
\[
\int_{\overline{\mathcal{M}}_{0,1}(X,d)}\ev_{1}^{*}(\Gamma_{1})\cdot\ctop(\mathcal{E}_{d,1})
\]
is supported in the automorphisms-free locus, we get Eq. (\ref{eq:WoutMarks}).
In order to prove the general case, it is enough to repeat the same
argument.
\end{proof}

\section{Equivariant classes in $\mathbb{P}^{n}$}

Let us recall \cite[Chapter 9]{cox1999mirror}. The standard action
of $T=(\mathbb{C}^{*})^{n+1}$ on $\mathbb{P}^{n}$ has $n+1$ fixed
points $\{q_{i}\}_{i=0}^{n}$, and $n(n+1)/2$ coordinate (invariant)
lines. We fix an isomorphism
\[
H^{*}(BT)\cong\mathbb{C}[\lambda_{0},\ldots,\lambda_{n}],
\]
where $\lambda_{i}$ is the corresponding weight of the $T$-action
on $\mathcal{O}_{\mathbb{P}^{n}}(-1)_{|q_{i}}$. There is an induced
$T$-action on $\overline{M}_{0,m}(\mathbb{P}^{n},d)$, with fixed
points given by stable maps $(C,f,p_{1},\ldots,p_{m})$ such that
each component of $C$ is either mapped by $f$ to some $q_{i}$ or
it is the cover of a coordinate line. Moreover marked points, nodes
of $C$ and ramification points of $f$ are mapped to fixed points.
In particular, if $C'$ is a component of $C$ and a degree $d'$
cover of a coordinate line via $f$, then locally $f$ is represented
by homogeneous coordinates $(x_{0},x_{1})\mapsto(x_{0}^{d'},x_{1}^{d'})$.

For each fixed stable map $(C,f,p_{1},\ldots,p_{m})$ we associate
a colored tree\footnote{A tree is a connected, undirected, simple graph.}
$\Gamma$. The vertices $v$ are in one-to-one correspondence with
the connected components $C_{v}$ of $f^{-1}(\{q_{0},\ldots,q_{n}\})$.
Each vertex $v$ is labeled by the point $q_{v}\in\{q_{i}\}_{i=0}^{n}$
such that $f(C_{v})=q_{v}$. Moreover, the edges correspond to those
irreducible components mapped by $f$ onto some coordinate line. We
associate to each edge $e$ the degree $d_{e}$ of the cover it represents.
We define the set $S_{v}$ consisting of those indexes $i\in\{1,\ldots,m\}$
for which $p_{i}\in C_{v}$. An isomorphism of such trees is an isomorphism
of graphs preserving the labels $q_{v},S_{v},d_{e}$.

The connected components of the fixed point locus of $\overline{M}_{0,m}(\mathbb{P}^{n},d)$
are parameterized by all possible trees $\Gamma$, modulo isomorphisms.
For each equivariant vector bundle $\mathcal{V}$ on $\overline{\mathcal{M}}_{0,m}(\mathbb{P}^{n},d)$,
and each polynomial $P(\mathcal{V})$ in the Chern classes of $\mathcal{V}$,
we will denote by $P^{T}(\mathcal{V})$ the corresponding equivariant
polynomial of $P(\mathcal{V})$. For example, $\ctopT(\mathcal{V})$
is the equivariant Euler class of $\mathcal{V}$. We use the following
notation: 
\begin{itemize}
\item $i_{\Gamma}:\overline{\mathcal{M}}_{\Gamma}\hookrightarrow\overline{\mathcal{M}}_{0,m}(\mathbb{P}^{n},d)$
is the substack of fixed maps with corresponding graph $\Gamma$,
\item $P^{T}(\mathcal{V})(\Gamma)=i_{\Gamma}^{*}\left(P^{T}(\mathcal{V})\right)$
is the pull back in equivariant cohomology 
\[
i_{\Gamma}^{*}:H_{T}^{*}(\overline{\mathcal{M}}_{0,m}(\mathbb{P}^{n},d))\otimes\mathbb{C}(\lambda_{0},\ldots,\lambda_{n})\rightarrow H_{T}^{*}(\overline{\mathcal{M}}_{\Gamma})\otimes\mathbb{C}(\lambda_{0},\ldots,\lambda_{n}),
\]
\item $G(d)$ is the set of all possible trees, modulo isomorphisms, corresponding
to fixed maps,
\item $V_{\Gamma}$ is the set of all vertices of $\Gamma$,
\item $E_{\Gamma}$ is the set of all edges of $\Gamma$,
\item for each $v\in V_{\Gamma}$, $\lambda_{v}\in\mathbb{C}[\lambda_{0},\ldots,\lambda_{n}]$
is the weight of the $T$-action on $\mathcal{O}_{\mathbb{P}^{n}}(-1)_{|q_{v}}$,
\item for each $e\in E_{\Gamma}$, $d_{e}$ and $(e(i),e(j))$ are, respectively,
the degree of the cover corresponding to $e$ and the pair of labels
of its vertices.
\end{itemize}
In order to compute the number of contact curves, we will use Bott
formula (see \cite{bott1967residue,kontsevich1995enumeration}):
\begin{equation}
\int_{\overline{\mathcal{M}}_{0,m}(\mathbb{P}^{n},d)}P(\mathcal{V})=\sum_{\Gamma\in G(d)}\frac{1}{a_{\Gamma}}\frac{P^{T}(\mathcal{V})(\Gamma)}{\ctopT(N_{\Gamma})(\Gamma)},\label{eq:Bott}
\end{equation}
where $N_{\Gamma}$ is the normal bundle of $\overline{\mathcal{M}}_{\Gamma}$,
and
\[
a_{\Gamma}:=|\mathrm{Aut}(\Gamma)|\cdot\prod_{e\in E_{\Gamma}}d_{e},
\]
where $|\mathrm{Aut}(\Gamma)|$ is the order of the group of automorphisms
of $\Gamma$.

Once we can compute $P^{T}(\mathcal{V})(\Gamma)$ and $a_{\Gamma}$,
we can apply formula (\ref{eq:Bott}) using a computer. In the case
$m=0$ (so, no marked points), in order to find $G(d)$ we have to
find all trees colored with $\{q_{0},\ldots,q_{n}\}$, and where each
edge $e$ has a weight $d_{e}\ge1$ such that $\sum_{e}d_{e}=d$.
For example, $G(1)$ contains the following trees:

{
\centering
\begin {tikzpicture}[auto ,node distance =2 cm,on grid , thick , state/.style ={ text=black }] 
\node[state] (A){$q_i$}; 
\node[state] (B) [right =of A] {$q_j$}; 
\path (A) edge node{$1$} (B); 
\node [] at (4.5,0.1) {$0\le i < j\le n,\,a_\Gamma =1.$}; 
\end{tikzpicture}
\par}

On the other hand, the set $G(2)$ contains the following trees:

{
\centering
\begin {tikzpicture}[auto ,node distance =2 cm,on grid , thick , state/.style ={  text=black }] 
\node[state] (A){$q_i$}; 
\node[state] (B) [right =of A] {$q_j$}; 
\node[state] (C) [right =of B] {$q_k$}; 
\path (A) edge node{$1$} (B); 
\path (B) edge node{$1$} (C); 
\node [ ] at (6 ,0.2) {$0\le i < k\le n$}; 
\node [below] at (6.5,0) {$0\le j\le n,\,j\neq i,k,\,a_\Gamma =1.$};
\end{tikzpicture} 

\begin {tikzpicture}[auto ,node distance =2 cm,on grid , thick , state/.style ={ text=black }] 
\node[state] (A){$q_i$}; 
\node[state] (B) [right =of A] {$q_j$}; 
\node[state] (C) [right =of B] {$q_i$}; 
\path (A) edge node{$1$} (B); 
\path (B) edge node{$1$} (C); 
\node [below] at (6.5 ,0.3) {$0\le i , j\le n,i\neq j,\,a_\Gamma =2.$};
\end{tikzpicture}

\begin {tikzpicture}[auto ,node distance =2 cm,on grid , thick , state/.style ={  text=black }] 
\node[state] (A){$q_i$}; 
\node[state] (B) [right =of A] {$q_j$};  
\path (A) edge node{$2$} (B);  
\node [] at (4.5 ,0) {$0\le i < j\le n,\,a_\Gamma =2.$}; 
\end{tikzpicture}
\par}

We will use Bott formula only in the case without marked points thanks
to Lemma \ref{lem:WoutMarks}. Since we are interested in the number
of contact curves meeting a certain number of linear subspaces, we
need to compute both the equivariant classes $\ctopT(\mathcal{E}_{d,0})$
and the equivariant class of the cycle of curves meeting a linear
subspace.

\subsection{Contact curves}

We will denote $\mathcal{E}_{d,0}$ by $\mathcal{E}_{d}$. The fiber
of $\mathcal{E}_{d}$ at a point $(C,f)$ of $\overline{M}_{0,0}(\mathbb{P}^{n},d)$
is the vector space $H^{0}(C,\omega_{C}\otimes f^{*}\mathcal{O}_{\mathbb{P}^{n}}(2))$,
where $\omega_{C}$ is the dualizing sheaf of $C$. There exists a
canonical isomorphism $H^{0}(C,\omega_{C}\otimes f^{*}\mathcal{O}_{\mathbb{P}^{n}}(2))\cong H^{1}(C,f^{*}\mathcal{O}_{\mathbb{P}^{n}}(-2))^{\vee}$.
So the weights of $H^{0}(C,\omega_{C}\otimes f^{*}\mathcal{O}_{\mathbb{P}^{n}}(2))$
are the weights of $H^{1}(C,f^{*}\mathcal{O}_{\mathbb{P}^{n}}(-2))$
multiplied by $(-1)$. 
\begin{lem}
\label{lem:AdditivityContact}Let $\Gamma\in G(d)$. Suppose that
there exist two trees $\Gamma_{1}\in G(d_{1})$, $\Gamma_{2}\in G(d_{2})$
such that $\Gamma$ is the union of $\Gamma_{1}$ and $\Gamma_{2}$
at a vertex $v$. Then
\begin{equation}
\ctopT(\mathcal{E}_{d})(\Gamma)=\ctopT(\mathcal{E}_{d_{1}})(\Gamma_{1})\,\ctopT(\mathcal{E}_{d_{2}})(\Gamma_{2})\,(2\lambda_{v}).\label{eq:AdditivityContact}
\end{equation}
\end{lem}

\begin{proof}
Let $(C,f)$ be a fixed stable map corresponding to $\Gamma$. Let
$(C_{1},f_{1},q'_{1})$ and $(C_{2},f_{2},q'_{2})$ be fixed stable
maps corresponding to, respectively, $\Gamma_{1}$ and $\Gamma_{2}$
such that the marked point is in $C_{v}$. In particular, $f_{1}(q'_{1})=f_{2}(q'_{2})=q_{v}$.
The standard gluing map 
\[
\overline{M}_{0,1}(\mathbb{P}^{n},d_{1})\times_{\mathbb{P}^{n}}\overline{M}_{0,1}(\mathbb{P}^{n},d_{2})\rightarrow\overline{M}_{0,0}(\mathbb{P}^{n},d)
\]
sends $((C_{1},f_{1},q'_{1}),(C_{2},f_{2},q'_{2}))$ to $(C,f)$.
Let us consider the normalization exact sequence tensorized by $\mathcal{O}_{\mathbb{P}^{n}}(-2)$
(see, e.g. \cite[pag 88]{arbarello2011geometry}):
\[
0\rightarrow f^{*}\mathcal{O}_{\mathbb{P}^{n}}(-2)\rightarrow f_{1}^{*}\mathcal{O}_{\mathbb{P}^{n}}(-2)\oplus f_{2}^{*}\mathcal{O}_{\mathbb{P}^{n}}(-2)\rightarrow\mathcal{O}_{\mathbb{P}^{n}}(-2)_{|q_{v}}\rightarrow0.
\]
If we take the dual of the exact sequence given by the spaces of global
sections, we get the following short exact sequence:
\begin{eqnarray*}
0\rightarrow & H^{1}(C_{1},f_{1}^{*}\mathcal{O}_{\mathbb{P}^{n}}(-2))^{\vee}\oplus H^{1}(C_{2},f_{2}^{*}\mathcal{O}_{\mathbb{P}^{n}}(-2))^{\vee} & \rightarrow\\
\rightarrow & H^{1}(C,f^{*}\mathcal{O}_{\mathbb{P}^{n}}(-2))^{\vee}\rightarrow H^{0}(q_{v},\mathcal{O}_{\mathbb{P}^{n}}(-2)_{|q_{v}})^{\vee} & \rightarrow0.
\end{eqnarray*}
This immediately implies Eq. (\ref{eq:AdditivityContact}).
\end{proof}
\begin{prop}
\label{prop:ClassE}Let $\Gamma\in G(d)$ be a tree. Then
\[
\ctopT(\mathcal{E}_{d})(\Gamma)=\left(\prod_{e\in E_{\Gamma}}\prod_{\alpha=1}^{2d_{e}-1}\frac{\alpha\lambda_{e(i)}+(2d_{e}-\alpha)\lambda_{e(j)}}{d_{e}}\right)\prod_{v\in V_{\Gamma}}(2\lambda_{v})^{\mathrm{val}(v)-1}
\]
where $\mathrm{val}(v)$ is the number of edges connected to $v$.
\end{prop}

Proposition \ref{prop:ClassE} was first proved in \cite[Proposição 3.1.3]{Eden}
with the following strategy. Let $f:\mathbb{P}^{1}\rightarrow\mathbb{P}^{3}$
be a stable curve of degree $d$. If we take the Euler exact sequence
of $\mathbb{P}^{1}$ \cite[II.8.13]{MR0463157} and tensor it by $\mathcal{O}_{\mathbb{P}^{1}}(2d)=f^{*}\mathcal{O}_{\mathbb{P}^{3}}(2)$,
we get the vector space $H^{0}(\mathbb{P}^{1},\omega_{\mathbb{P}^{1}}\otimes f^{*}\mathcal{O}_{\mathbb{P}^{3}}(2))$
as the kernel of the surjective map 
\begin{equation}
H^{0}(\mathbb{P}^{1},\mathcal{O}_{\mathbb{P}^{1}}(2d-1))^{\oplus2}\twoheadrightarrow H^{0}(\mathbb{P}^{1},\mathcal{O}_{\mathbb{P}^{1}}(2d)).\label{eq:Amorim}
\end{equation}
He found the weights of the action on $H^{0}(\mathbb{P}^{1},\omega_{\mathbb{P}^{1}}\otimes f^{*}\mathcal{O}_{\mathbb{P}^{3}}(2))$
from the weights of $H^{0}(\mathbb{P}^{1},\mathcal{O}_{\mathbb{P}^{1}}(2d-1))^{\oplus2}$
and $H^{0}(\mathbb{P}^{1},\mathcal{O}_{\mathbb{P}^{1}}(2d))$, and
from the map in (\ref{eq:Amorim}).
\begin{proof}[Proof of Proposition \ref{prop:ClassE}]
Suppose that $\Gamma=e$ is the graph with only one edge and labels
$(q_{i},q_{j})$ corresponding to a degree $d$ map $(\mathbb{P}^{1},f)$.
One can show that a basis for $H^{1}(\mathbb{P}^{1},f^{*}\mathcal{O}_{\mathbb{P}^{n}}(-2))$
is given by the \v{C}ech cocycles:
\begin{equation}
z_{0}^{-\alpha}z_{1}^{-(2d-\alpha)},1\le\alpha\le2d-1,\label{eq:cocycles}
\end{equation}
see \cite[III.5.5.1]{MR0463157}. We know that $x_{0},x_{1}\in H^{0}(\mathbb{P}^{1},\mathcal{O}_{\mathbb{P}^{1}}(1))$
have weights $\lambda_{i},\lambda_{j}$. Since $x_{k}=z_{k}^{d}$
via $f$, it follows that $z_{0},z_{1}$ have weights $\frac{\lambda_{i}}{d},\frac{\lambda_{j}}{d}$.
Hence the cocycles in (\ref{eq:cocycles}) have weights 
\[
(-1)\frac{\alpha\lambda_{i}+(2d-\alpha)\lambda_{j}}{d},1\le\alpha\le2d-1.
\]
So the equivariant Euler class of $\mathcal{E}_{d}$ is
\[
\prod_{\alpha=1}^{2d-1}(-1)(-1)\frac{\alpha\lambda_{i}+(2d-\alpha)\lambda_{j}}{d}=\prod_{\alpha=1}^{2d-1}\frac{\alpha\lambda_{i}+(2d-\alpha)\lambda_{j}}{d}.
\]
We can prove the proposition for any $\Gamma$ by induction on the
number of edges and using Lemma \ref{lem:AdditivityContact}. This
concludes the proof.
\end{proof}

\subsection{Incidence to a linear subspace}

The cycle of stable maps whose image meets a general codimension $r+1$
linear subspace of $\mathbb{P}^{n}$ is known to be $\pi_{*}(h^{r+1})$
where $h=c_{1}(\ev_{1}^{*}(\mathcal{O}_{\mathbb{P}^{n}}(1)))$. In
this section, we will represent $\pi_{*}(h^{r+1})$ as a linear combination
of polynomials in the Chern classes of vector bundles. It follows
that the equivariant class of $\pi_{*}(h^{r+1})$, which we denote
by $[\pi_{*}(h^{r+1})]^{T}$, is the linear combination of the equivariant
version of the same polynomials. At the very end the result will be
the following.
\begin{thm}
\label{thm:Equivariant-class-Linear}Let $\Gamma\in G(d)$. The equivariant
class for the incidence to a codimension $r+1$ linear subspace of
$\mathbb{P}^{n}$ in the component $\overline{\mathcal{M}}_{\Gamma}$
is
\[
\sum_{e\in E_{\Gamma}}d_{e}\sum_{t=0}^{r}\lambda_{e(i)}^{t}\lambda_{e(j)}^{r-t}.
\]
\end{thm}

See \cite[Proposição 3.1.2]{Eden} for the case $r=1$, with a similar
strategy.
\begin{lem}
\label{lem:DetMatrixV}Let $d,r$ be integers. Let $V$ be the matrix
of order $r+1$ whose coefficients are $(V)_{a,k}=a^{k}$. Let $p(x)=\sum_{i}p_{i}x^{i}$
be a polynomial of degree at most $r+1$ with $p(0)=0$. Let $V_{p(x)}$
be the matrix of order $r+1$ such that $(V_{p(x)})_{a,k}=(V)_{a,k}$
for $k\neq r+1$, and $(V_{p(x)})_{a,r+1}=p(ad)$. That is,
\[
V_{p(x)}:=\left(\begin{array}{ccccccc}
1 & 1^{2} & 1^{3} & 1^{4} & \ldots & 1^{r} & p(d)\\
2 & 2^{2} & 2^{3} & 2^{4} & \ldots & 2^{r} & p(2d)\\
\vdots & \vdots & \vdots & \vdots & \vdots & \vdots & \vdots\\
r+1 & (r+1)^{2} & (r+1)^{3} & (r+1)^{4} & \ldots & (r+1)^{r} & p((r+1)d)
\end{array}\right).
\]
Then $\det V_{p(x)}=d^{r+1}p_{r+1}\det V.$
\end{lem}

\begin{proof}
If $p(x)=p_{k}x^{k}$ is a monomial with $1\le k\le r$, we see that
the last column of $V_{p(x)}$ is a multiple of the $k^{\mathrm{th}}$
column, then $\det V_{p(x)}=0$. So can suppose that $p(x)=p_{r+1}x^{r+1}$,
and the result follows easily.
\end{proof}
\begin{lem}
The cycle $\pi_{*}(h^{r+1})$ is a linear combination of Chern characters
of equivariant vector bundles.
\end{lem}

\begin{proof}
Let $a$ be an integer, consider the sheaf $\pi_{*}(\mathcal{L}(a))$
on $\overline{\mathcal{M}}_{0,0}(\mathbb{P}^{n},d)$, where $\mathcal{L}(a):=\ev_{0,0}^{*}(\mathcal{O}_{\mathbb{P}^{n}}(a))$.
Given a stable map $(C,f)$, using Riemann-Roch we can see that $h^{1}(C,f^{*}\mathcal{O}_{\mathbb{P}^{n}}(a))=0$
for every $a\ge0$, hence for those values of $a$ we have that $\pi_{*}(\mathcal{L}(a))$
is a vector bundle and $\pi_{!}(\mathcal{L}(a))=\pi_{*}(\mathcal{L}(a))$.
Let $\td(T_{\pi})$ be the Todd's class of the coherent sheaf $T_{\pi}$,
and denote by $\td(T_{\pi})_{k}$ its $k$-codimensional component.
We apply Grothendieck-Riemann-Roch Theorem (GRR) \cite[Chapter 18]{MR1644323}
(see also \cite[3E]{harris2006moduli})
\begin{eqnarray*}
\ch(\pi_{*}\mathcal{L}(a)) & = & \pi_{*}(\ch(\mathcal{L}(a))\cdot\td(T_{\pi}))\\
 & = & \pi_{*}\left(\left(1+ah+\frac{(ah)^{2}}{2!}+\ldots\right)\left(\td(T_{\pi})_{0}+\td(T_{\pi})_{1}+\td(T_{\pi})_{2}+\ldots\right)\right)\\
 & = & \pi_{*}\left(\left[\td(T_{\pi})_{0}\right]+\left[\td(T_{\pi})_{1}+ah\right]+\ldots\right).
\end{eqnarray*}
Moreover, since $\ch_{i}(\pi_{*}\mathcal{L}(0))=0$ for $i\ge1$,
by GRR
\begin{eqnarray*}
\ch(\pi_{*}\mathcal{L}(0)) & = & \pi_{*}(\ch(\mathcal{L}(0))\cdot\td(T_{\pi}))\\
 & = & \pi_{*}(\td(\mathcal{L}(0)\otimes T_{\pi}))\\
 & = & \pi_{*}(\td(T_{\pi})),
\end{eqnarray*}
it follows that $\pi_{*}(\td(T_{\pi})_{i})=0$ for $i\ge1$. The component
$\ch_{r}(\pi_{*}\mathcal{L}(a))$ is
\begin{eqnarray}
\ch_{r}(\pi_{*}\mathcal{L}(a)) & = & \underset{=0}{\underbrace{\pi_{*}(\td(T_{\pi})_{r+1})}}+\sum_{k=1}^{r+1}a^{k}\cdot\pi_{*}\left(\frac{\td(T_{\pi})_{r+1-k}h^{k}}{k!}\right)\nonumber \\
 & = & \sum_{k=1}^{r+1}a^{k}\cdot\pi_{*}\left(\frac{\td(T_{\pi})_{r+1-k}h^{k}}{k!}\right).\label{eq:LinearSys}
\end{eqnarray}
Now we use a little bit of linear algebra. In (\ref{eq:LinearSys}),
we consider $\pi_{*}\left(\frac{\td(T_{\pi})_{r+1-k}h^{k}}{k!}\right)$
as unknowns, so for $a=1,2,3,\ldots,r+1$ we get $r+1$ linearly independent
equations. We can apply Cramer's rule and get an explicit expression
\[
\frac{1}{(r+1)!}\pi_{*}(h^{r+1})=\frac{1}{\det V}\det\left(\begin{array}{ccccc}
1 & 1^{2} & \ldots & 1^{r} & \ch_{r}(\pi_{*}\mathcal{L}(1))\\
2 & 2^{2} & \ldots & 2^{r} & \ch_{r}(\pi_{*}\mathcal{L}(2))\\
\vdots & \vdots & \vdots & \vdots & \vdots\\
r+1 & (r+1)^{2} & \ldots & (r+1)^{r} & \ch_{r}(\pi_{*}\mathcal{L}(r+1))
\end{array}\right),
\]
where $V$ is the matrix of the linear system, i.e., $(V)_{a,k}=a^{k}$
for $a,k=1,\ldots,r+1$.
\end{proof}
\begin{lem}
\label{lem:AdditivityLinear}Let $\Gamma\in G(d)$. Suppose that there
exist two trees $\Gamma_{1}\in G(d_{1})$, $\Gamma_{2}\in G(d_{2})$
such that $\Gamma$ is the union of $\Gamma_{1}$ and $\Gamma_{2}$.
Then
\begin{equation}
[\pi_{*}(h^{r+1})]^{T}(\Gamma)=[\pi_{*}(h^{r+1})]^{T}(\Gamma_{1})+[\pi_{*}(h^{r+1})]^{T}(\Gamma_{2}).\label{eq:AdditivityLinear}
\end{equation}
\end{lem}

\begin{proof}
We will argue like in Lemma \ref{lem:AdditivityContact}. Let $v$
be the common vertex of $\Gamma_{1}$ and $\Gamma_{2}$. From the
normalization exact sequence tensorized by $\mathcal{O}_{\mathbb{P}^{n}}(a)$,
we get the short exact sequence
\[
\begin{array}{c}
0\rightarrow H^{0}(C,f^{*}\mathcal{O}_{\mathbb{P}^{n}}(a))\rightarrow H^{0}(C_{1},f_{1}^{*}\mathcal{O}_{\mathbb{P}^{n}}(a))\oplus H^{0}(C_{2},f_{2}^{*}\mathcal{O}_{\mathbb{P}^{n}}(a))\\
\qquad\qquad\qquad\qquad\qquad\qquad\qquad\qquad\qquad\qquad\rightarrow H^{0}(q_{v},\mathcal{O}_{\mathbb{P}^{n}}(a)_{|q_{v}})\rightarrow0.
\end{array}
\]
This immediately implies that
\[
\ch_{r}^{T}(\pi_{*}\mathcal{L}(a))(\Gamma)=\ch_{r}^{T}(\pi_{*}\mathcal{L}(a))(\Gamma_{1})+\ch_{r}^{T}(\pi_{*}\mathcal{L}(a))(\Gamma_{2})-\ch_{r}^{T}(\pi_{*}\mathcal{L}(a))(v).
\]
Since $\pi_{*}(h^{r+1})$ is a linear combination of $\{\ch_{r}(\pi_{*}\mathcal{L}(a))\}_{a=1}^{r+1}$,
in order to get $[\pi_{*}(h^{r+1})]^{T}(\Gamma)$ we just need to
substitute in the expression of $\pi_{*}(h^{r+1})$ all $\ch_{r}(\pi_{*}\mathcal{L}(a))$
by $\ch_{r}^{T}(\pi_{*}\mathcal{L}(a))(\Gamma)$. Since
\[
\ch_{r}^{T}(\pi_{*}\mathcal{L}(a))(v)=a^{r}\lambda_{v}^{r},
\]
by Lemma \ref{lem:DetMatrixV} we see that $\ch_{r}^{T}(\pi_{*}\mathcal{L}(a))(v)$
has no influence on the value of $[\pi_{*}(h^{r+1})]^{T}(\Gamma)$,
so by linearity we get Eq. (\ref{eq:AdditivityLinear}).
\end{proof}
Finally, we can prove the theorem.
\begin{proof}[Proof of Theorem \ref{thm:Equivariant-class-Linear}.]
Suppose that $\Gamma=e$ is the graph with only one edge and labels
$(q_{i},q_{j})$ corresponding to a degree $d$ map $(\mathbb{P}^{1},f)$.
The equivariant polynomial $r!\ch_{r}^{T}(\pi_{*}\mathcal{L}(a))(\Gamma)$
is given by the sum of the $r$-powers of all weights of the action
on $H^{0}(\mathbb{P}^{1},f^{*}\mathcal{O}_{\mathbb{P}^{n}}(a))$.
That is
\begin{eqnarray}
\ch_{r}^{T}(\pi_{*}\mathcal{L}(a))(\Gamma) & = & \frac{1}{r!}\sum_{k=0}^{ad}\left(\frac{(ad-k)\lambda_{i}+k\lambda_{j}}{d}\right)^{r}\nonumber \\
 & = & \frac{1}{d^{r}}\frac{1}{r!}\sum_{k=0}^{ad}\left(\sum_{t=0}^{r}\dbinom{r}{t}(ad-k)^{t}k^{r-t}\lambda_{i}^{t}\lambda_{j}^{r-t}\right)\nonumber \\
 & = & \frac{1}{d^{r}}\frac{1}{r!}\sum_{t=0}^{r}\dbinom{r}{t}\left(\sum_{k=0}^{ad}(ad-k)^{t}k^{r-t}\right)\lambda_{i}^{t}\lambda_{j}^{r-t}.\label{eq:linear}
\end{eqnarray}
Equation (\ref{eq:linear}) tells us that each $\ch_{r}^{T}(\pi_{*}\mathcal{L}(a))(\Gamma)$
is a linear combination of monomials $\lambda_{i}^{t}\lambda_{j}^{r-t}$
for $t=0,\ldots,r$. This implies that also $[\pi_{*}(h^{r+1})]^{T}(\Gamma)$
is a linear combination of the same monomials. By linearity, the coefficient
of the monomial $\lambda_{i}^{t}\lambda_{j}^{r-t}$ in $[\pi_{*}(h^{r+1})]^{T}(\Gamma)$
is
\begin{equation}
\frac{1}{d^{r}}\frac{1}{r!}\frac{(r+1)!}{\det V}\det\left(\begin{array}{cccc}
1 & \ldots & 1^{r} & \dbinom{r}{t}\sum_{k=0}^{d}(d-k)^{t}k^{r-t}\\
2 & \ldots & 2^{r} & \dbinom{r}{t}\sum_{k=0}^{2d}(2d-k)^{t}k^{r-t}\\
\vdots & \vdots & \vdots & \vdots\\
r+1 & \ldots & (r+1)^{r} & \dbinom{r}{t}\sum_{k=0}^{(r+1)d}((r+1)d-k)^{t}k^{r-t}
\end{array}\right).\label{eq:CoefLambdaiLambdaj}
\end{equation}
In order to prove the theorem, we need to show that the value of (\ref{eq:CoefLambdaiLambdaj})
is $d$.
\begin{claim}
\label{claim:p}There exists a degree $r+1$ polynomial $p(x)=\sum_{i}p_{i}x^{i}\in\mathbb{Q}[x]$
such that $p(0)=0$, $p_{r+1}=(r+1)^{-1}$ and 
\begin{equation}
p(ad)=\dbinom{r}{t}\sum_{k=0}^{ad}(ad-k)^{t}k^{r-t},1\le a\le r+1.\label{eq:p(ad)}
\end{equation}
\end{claim}

\begin{proof}
For every integer $q\ge0$, let us denote by $S^{q}(x)$ the Faulhaber\footnote{See \cite{bernoulli1709dissertatio}, or \cite[pag 503]{Knuth} for
a modern reference.} polynomial of degree $q+1$. That is, a polynomial with leading coefficient
$(q+1)^{-1}$ such that for every integer $n\ge0$, $S^{q}(n)=\sum_{k=1}^{n}k^{q}$.
If we define
\[
p(x):=\dbinom{r}{t}\sum_{j=0}^{t}(-1)^{j}\dbinom{t}{j}x^{t-j}S^{r+j-t}(x),
\]
it is clear that $p(x)$ has degree $r+1$ and $p(0)=0$. Moreover,
it satisfies (\ref{eq:p(ad)}), indeed:
\begin{eqnarray*}
p(ad) & = & \dbinom{r}{t}\sum_{j=0}^{t}(-1)^{j}\dbinom{t}{j}(ad)^{t-j}S^{r+j-t}(ad)\\
 & = & \dbinom{r}{t}\sum_{j=0}^{t}(-1)^{j}\dbinom{t}{j}(ad)^{t-j}\left(\sum_{k=0}^{ad}k^{r+j-t}\right)\\
 & = & \dbinom{r}{t}\sum_{k=0}^{ad}\left(\sum_{j=0}^{t}(-1)^{j}\dbinom{t}{j}(ad)^{t-j}k^{r+j-t}\right)\\
 & = & \dbinom{r}{t}\sum_{k=0}^{ad}\left(\sum_{j=0}^{t}(-1)^{j}\dbinom{t}{j}(ad)^{t-j}k^{j}\right)k^{r-t}\\
 & = & \dbinom{r}{t}\sum_{k=0}^{ad}(ad-k)^{t}k^{r-t}.
\end{eqnarray*}
Let us compute the leading coefficient of $p(x)$:
\begin{eqnarray*}
p_{r+1} & = & \dbinom{r}{t}\sum_{j=0}^{t}(-1)^{j}\dbinom{t}{j}\frac{1}{r+j-t+1}\\
 & = & \frac{1}{r+1}\sum_{j=0}^{t}(-1)^{j}\dbinom{r}{t}\dbinom{t}{j}\frac{r+1}{r+j-t+1}\\
 & = & \frac{1}{r+1}\sum_{j=0}^{t}(-1)^{j}\dbinom{r-t+j}{j}\dbinom{r+1}{t-j}.
\end{eqnarray*}
If we apply Eq. (17) and (21) of \cite[§1.2.6]{Knuth}, we get
\begin{eqnarray*}
\frac{1}{r+1}\sum_{j=0}^{t}(-1)^{j}\dbinom{r-t+j}{j}\dbinom{r+1}{t-j} & = & \frac{1}{r+1}\sum_{j=0}^{t}\dbinom{t-r-1}{j}\dbinom{r+1}{t-j}\\
 & = & \frac{1}{r+1}\dbinom{t}{t},
\end{eqnarray*}
that is, $p_{r+1}=(r+1)^{-1}$. Hence, the claim is proved.
\end{proof}
Following the notation of Lemma \ref{lem:DetMatrixV}, we can denote
the matrix appearing in (\ref{eq:CoefLambdaiLambdaj}) by $V_{p(x)}$,
where $p(x)$ is the polynomial of Claim \ref{claim:p}. The value
of (\ref{eq:CoefLambdaiLambdaj}) is: 
\begin{eqnarray*}
\frac{1}{d^{r}}\frac{1}{r!}\frac{(r+1)!}{\det V}\det V_{p(x)} & = & \frac{1}{d^{r}}\frac{r+1}{\det V}d^{r+1}\frac{1}{r+1}\det V\\
 & = & d.
\end{eqnarray*}
Finally, 
\[
[\pi_{*}(h^{r+1})]^{T}(\Gamma)=d(\lambda_{i}^{r}+\lambda_{i}^{r-1}\lambda_{j}+\lambda_{i}^{r-2}\lambda_{j}^{2}+\lambda_{i}^{r-3}\lambda_{j}^{3}+\ldots+\lambda_{i}\lambda_{j}^{r-1}+\lambda_{j}^{r}).
\]
We can prove the theorem for any $\Gamma$ by induction on the number
of edges and using Lemma \ref{lem:AdditivityLinear}. This concludes
the proof.
\end{proof}
\begin{rem}
Note that by linearity of push forward, the equivariant class of curves
meeting a subvariety of codimension $r$ and degree $k$ is $k$-times
the class of the linear subspace of the same codimension. For this
reason, it is not restrictive to consider only linear subspaces in
the enumeration of contact curves.
\end{rem}

\section{Final remarks}

The explicit number of contact curves of degree $d$ in $\mathbb{P}^{2n+1}$
can be calculated combining the result of the previous sections. We
know, by Example \ref{exa:SL}, that the subscheme of contact lines
is a section of the Plücker line bundle of the Grassmannian of lines.
So all characteristic numbers of contact lines can be deduced by classical
Schubert calculus. For curves of higher degree, no generating formula
is known.
\begin{example}
Suppose we want to compute the number of contact lines in $\mathbb{P}^{3}$,
meeting $a_{2}$ lines and $a_{3}$ points. By Corollary \ref{cor:Enum},
Lemma \ref{lem:WoutMarks} and Equation (\ref{eq:Bott}), we know
that
\[
a_{2}+2a_{3}=3,
\]
and the number of such lines is
\begin{eqnarray*}
N_{1}(a_{2},a_{3}) & = & \int_{\overline{M}_{0,0}(\mathbb{P}^{3},1)}c_{\mathrm{top}}(\mathcal{E}_{d})(\pi_{*}(h^{2}))^{a_{2}}(\pi_{*}(h^{3}))^{a_{3}}\\
 & = & \sum_{\Gamma\in G(1)}\frac{1}{a_{\Gamma}}\frac{\ctopT(\mathcal{E}_{1})(\Gamma)\left([\pi_{*}(h^{2})]^{T}(\Gamma)\right)^{a_{2}}\left([\pi_{*}(h^{3})]^{T}(\Gamma)\right)^{a_{3}}}{\ctopT(N_{\Gamma})(\Gamma)}\\
 & = & \sum_{0\le i<j\le3}\frac{(\lambda_{i}+\lambda_{j})(\lambda_{i}+\lambda_{j})^{a_{2}}(\lambda_{i}^{2}+\lambda_{i}\lambda_{j}+\lambda_{j}^{2})^{a_{3}}}{\prod_{k\neq i,j}(\lambda_{i}-\lambda_{k})(\lambda_{j}-\lambda_{k})}.
\end{eqnarray*}
The value of $\ctopT(N_{\Gamma})(\Gamma)$ is known thanks to \cite{kontsevich1995enumeration}.
We expect that the above sum equals an integer number. For curves
of higher degree, we argue in the same way.
\end{example}

In Table \ref{tab:EnuNum}, we listed the number of rational contact
curves of degree $d$, denoted by $N_{d}(a_{2},a_{3})$, for all possible
values of $(a_{2},a_{3})$ and $d\le3$, computed using Bott formula
and implemented in \textit{Wolfram Mathematica}. Anyway, our results
can be generalized \textit{mutatis mutandis} to every odd dimensional
projective space. To my knowledge, characteristic numbers of contact
curves in spaces different from $\mathbb{P}^{3}$ have never been
investigated.

The code used here is available on author's website, or can be provided
upon request.

\begin{table}[h]
\centering%
\begin{tabular}{|c|c|c||c|c|c||c|c|c|}
\hline 
$d$ &
$(a_{2},a_{3})$ &
$N_{1}(a_{2},a_{3})$ &
$d$ &
$(a_{2},a_{3})$ &
$N_{2}(a_{2},a_{3})$ &
$d$ &
$(a_{2},a_{3})$ &
$N_{3}(a_{2},a_{3})$\tabularnewline
\hline 
\hline 
\multirow{2}{*}{$1$} &
$(3,0)$ &
$2$ &
\multirow{3}{*}{$2$} &
$(5,0)$ &
$40$ &
\multirow{4}{*}{$3$} &
$(7,0)$ &
$4160$\tabularnewline
\cline{2-3} \cline{3-3} \cline{5-6} \cline{6-6} \cline{8-9} \cline{9-9} 
 & $(1,1)$ &
$1$ &
 & $(3,1)$ &
$8$ &
 & $(5,1)$ &
$512$\tabularnewline
\cline{1-3} \cline{2-3} \cline{3-3} \cline{5-6} \cline{6-6} \cline{8-9} \cline{9-9} 
\multicolumn{1}{c}{} &
\multicolumn{1}{c}{} &
\multicolumn{1}{c|}{} &
 & $(1,2)$ &
$2$ &
 & $(3,2)$ &
$72$\tabularnewline
\cline{4-6} \cline{5-6} \cline{6-6} \cline{8-9} \cline{9-9} 
\multicolumn{1}{c}{} &
\multicolumn{1}{c}{} &
\multicolumn{1}{c}{} &
\multicolumn{1}{c}{} &
\multicolumn{1}{c}{} &
\multicolumn{1}{c|}{} &
 & $(1,3)$ &
$12$\tabularnewline
\cline{7-9} \cline{8-9} \cline{9-9} 
\end{tabular}\caption{\label{tab:EnuNum}Number of contact curves in $\mathbb{P}^{3}$ of
degree $d\le3$.}
\end{table}

The values of $N_{2}(5,0)$ and $N_{3}(7,0)$ were already found in
\cite{levcovitz2011symplectic,Eden}. One may try to get the number
of smooth contact curves of some degree by throwing out the reducible
ones. For example, contact conics are union of lines (see, e.g., \cite[Proposition 17]{LM07}).
Hence, also reducible contact cubics must be union of lines. Using
the fact that $N_{1}(1,1)=1$, it can be easily shown that there are
exactly $9$ reducible contact cubics through $3$ points and a line.
Since $N_{3}(1,3)=12$, it means that we have exactly $3$ smooth
contact cubics through these points and that line, as proved in \cite[Corollary 6.3]{kal}
with other techniques. The problem of finding the number of smooth
contact curves is interesting in its own and we hope to address these
questions elsewhere.

In Table \ref{tab:EnuNuminP5} and \ref{tab:EnuNuminP5-con} we listed
the number of contact curves in $\mathbb{P}^{5}$ of degree $1$ and
$2$ meeting $a_{i}$ linear subspaces of codimension $i$ for $i=2,3,4,5$.

\begin{table}[H]
\centering%
\begin{tabular}{|c|c||c|c|}
\hline 
$a=(a_{2},a_{3},a_{4},a_{5})$ &
$N_{1}(a)$ &
$(a_{2},a_{3},a_{4},a_{5})$ &
$N_{1}(a)$\tabularnewline
\hline 
\hline 
$(7,0,0,0)$ &
$14$ &
$(1,3,0,0)$ &
$4$\tabularnewline
\hline 
$(5,1,0,0)$ &
$9$ &
$(1,1,0,1)$ &
$1$\tabularnewline
\hline 
$(4,0,1,0)$ &
$4$ &
$(1,0,2,0)$ &
$2$\tabularnewline
\hline 
$(3,2,0,0)$ &
$6$ &
$(0,2,1,0)$ &
$2$\tabularnewline
\hline 
$(2,1,1,0)$ &
$3$ &
$(0,0,1,1)$ &
$1$\tabularnewline
\hline 
\end{tabular}\caption{\label{tab:EnuNuminP5}Number of contact curves in $\mathbb{P}^{5}$
of degree $1$.}
\end{table}
\begin{table}[H]
\centering%
\begin{tabular}{|c|c||c|c||c|c|}
\hline 
$a=(a_{2},a_{3},a_{4},a_{5})$ &
$N_{2}(a)$ &
$(a_{2},a_{3},a_{4},a_{5})$ &
$N_{2}(a)$ &
$(a_{2},a_{3},a_{4},a_{5})$ &
$N_{2}(a)$\tabularnewline
\hline 
\hline 
$(11,0,0,0)$ &
$103876$ &
$(4,2,1,0)$ &
$624$ &
$(1,5,0,0)$ &
$288$\tabularnewline
\hline 
$(9,1,0,0)$ &
$30864$ &
$(4,0,1,1)$ &
$541$ &
$(1,3,0,1)$ &
$28$\tabularnewline
\hline 
$(8,0,1,0)$ &
$5798$ &
$(3,4,0,0)$ &
$912$ &
$(1,2,2,0)$ &
$48$\tabularnewline
\hline 
$(7,2,0,0)$ &
$9420$ &
$(3,2,0,1)$ &
$76$ &
$(1,1,0,2)$ &
$4$\tabularnewline
\hline 
$(7,0,0,1)$ &
$544$ &
$(3,1,2,0)$ &
$152$ &
$(1,0,2,1)$ &
$8$\tabularnewline
\hline 
$(6,1,1,0)$ &
$1898$ &
$(3,0,0,2)$ &
$8$ &
$(0,4,1,0)$ &
$64$\tabularnewline
\hline 
$(5,3,0,0)$ &
$2924$ &
$(2,3,1,0)$ &
$200$ &
$(0,2,1,1)$ &
$8$\tabularnewline
\hline 
$(5,1,0,1)$ &
$202$ &
$(2,1,1,1)$ &
$22$ &
$(0,1,3,0)$ &
$12$\tabularnewline
\hline 
$(5,0,2,0)$ &
$436$ &
$(2,0,3,0)$ &
$44$ &
$(0,0,1,2)$ &
$2$\tabularnewline
\hline 
\end{tabular}\caption{\label{tab:EnuNuminP5-con}Number of contact curves in $\mathbb{P}^{5}$
of degree $2$.}
\end{table}
\bibliographystyle{amsalpha}
\bibliography{0C__Users_magno_Dropbox_Universita_Ricerche_Contact_Curves_Latex_RefContactCurves}

\end{document}